\documentclass[11pt,a4paper]{amsart}
\usepackage[margin=30mm]{geometry}
\usepackage{amsmath,amssymb,amsthm,mathtools}
\usepackage{enumitem}
\usepackage{xcolor}
\usepackage{hyperref}
\usepackage{tikz}
\hypersetup{colorlinks=true,linkcolor=blue,citecolor=blue,urlcolor=blue}

\newtheorem{theorem}{Theorem}[section]
\newtheorem{proposition}[theorem]{Proposition}
\newtheorem{lemma}[theorem]{Lemma}
\newtheorem{corollary}[theorem]{Corollary}

\theoremstyle{definition}
\newtheorem{definition}[theorem]{Definition}

\theoremstyle{remark}
\newtheorem{remark}[theorem]{Remark}

\newcommand{\RR}{\mathbb R}
\newcommand{\Tc}{{\mathcal T}}
\newcommand{\Fc}{{\mathcal F}}
\newcommand{\conv}{\operatorname{conv}}

\newcommand{\eb}{{\bf e}}

\title[Minkowski decomposability of symmetric edge polytopes]
{Minkowski decomposability of symmetric edge polytopes}

\author{Akihiro Higashitani}
\address{Department of Pure and Applied Mathematics, Graduate School of Information Science and Technology, Osaka University, Suita, Osaka 565-0871, Japan}
\email{higashitani@ist.osaka-u.ac.jp}

\author{Aki Mori}
\address{Center for Physics and Mathematics, Institute for Liberal Arts and Science, Osaka Electro-communication University, Neyagawa, Osaka 572-8530, Japan}
\email{a-mori@osakac.ac.jp}

\begin{document}
\begin{abstract}
In this paper, we study the Minkowski decomposability of symmetric edge polytopes $P_G^\pm$ of a finite simple graph $G$ on vertex set $[n]$. 
More precisely, we give a complete characterization of graphs whose symmetric edge polytopes are Minkowski decomposable. 
We prove that $P_G^\pm$ is Minkowski decomposable if and only if 
$G$ is one of the three complete multipartite graphs: $K_n$, $K_{2,n-2}$, or $K_{1,1,n-2}$. 
In other words, if $G$ does not belong to these three families, then $P_G^\pm$ is Minkowski indecomposable. 
\end{abstract}

\maketitle

\section{Introduction}

The {\emph{Minkowski sum} of two convex polytopes $Q,R\subset \RR^d$ is
\[
  Q+R=\{q+r:q\in Q,\ r\in R\}.
\]
A polytope $P$ is called \emph{Minkowski decomposable} if $P=Q+R$ for some polytopes $Q$ and $R$ neither of which is homothetic to $P$, 
where we say that two polytopes are homothetic if they are related via translation and positive dilations. 
Otherwise $P$ is called \emph{Minkowski indecomposable}. 
Minkowski decompositions are one of the basic operations in convex geometry. 
They also appear, with different terminology, in related areas: the cone of summands is closely related to \emph{type cones} in convex geometry, 
\emph{nef cones} in algebraic geometry, and \emph{deformation cones} in geometric combinatorics. 
The extremal rays of such cones correspond to homothety classes of Minkowski indecomposable polytopes, 
and recognizing such polytopes is a fundamental structural problem in several areas. 

Regarding Minkowski indecomposability of polytopes, a classical starting point is due to Shephard~\cite{Shephard1963}, 
who developed criteria in terms of strong chains of Minkowski indecomposable faces. 
McMullen later gave a useful weakening of Shephard's criterion, involving strongly connected families of Minkowski indecomposable faces touching all facets~\cite{McM}, 
which is the criterion we use in the proof of our main theorem. 
These criteria are particularly effective when many low-dimensional faces are already known to be Minkowski indecomposable, 
for example when suitable triangular faces can be organized into a strongly connected family. 
More recently, Minkowski indecomposability has been studied from a more systematic viewpoint through deformation cones and edge-deformation relations. 
A major recent advance in this direction is the work of Padrol and Poullot~\cite{PP}, who introduced the graph of implicit edge dependencies. 
Their framework unifies and extends many previous indecomposability criteria and also applies beyond indecomposability. 
A powerful application of this framework is the study of indecomposability of $0/1$-polytopes provided in ~\cite{HPS}. 
In particular, it was proved in~\cite{HPS} that every $0/1$-polytope decomposes uniquely as a Cartesian product of proper indecomposable $0/1$-polytopes. 
This paper is motivated by the same general question, but for a different and very active class of lattice polytopes: symmetric edge polytopes. 
Unlike $0/1$-polytopes, the entries of their vertices are given by $-1,0$ or $1$. 

Let $G$ be a finite simple graph on vertex set $[n]:=\{1,\ldots,n\}$. 
We regard $G$ as a directed graph having two orientations $ij$ and $ji$ for each edge $\{i,j\} \in E(G)$. 
Let \[  \rho(ij)=\eb_i-\eb_j, \]
where $\eb_1,\ldots,\eb_n$ are the unit vectors of $\RR^n$. 
Thus, $\rho(ji)=-\rho(ij)$.  The \emph{symmetric edge polytope} of $G$ is
\[  P_G^{\pm}=\conv\{\rho(ij): ij \text{ is a directed edge of }G\}  =\conv\{\pm(\eb_i-\eb_j):\{i,j\}\in E(G)\}.\]
This polytope lies in the hyperplane $\{x_1+\cdots+x_n=0\} \subset \RR^n$. 
It is known that the symmetric edge polytope of a connected graph with $n$ vertices is a  $(n-1)$-dimensional polytope.
Symmetric edge polytopes have become a very active class in Ehrhart theory and related areas of lattice polytope theory. 
Their reflexivity, unimodular triangulations, and arithmetic properties have been studied from several viewpoints; see, e.g.,~\cite{HJM,OT,DJKV}. 
In particular, Ohsugi--Tsuchiya formulated a conjecture on the $\gamma$-positivity for symmetric edge polytopes~\cite{OT}, 
and D'Al{\`i}--Juhnke-Kubitzke--K{\"o}hne--Venturello obtained several partial results~\cite{DJKV}. 
Recently, however, Ferroni disproved this conjecture by constructing an infinite family of counterexamples~\cite{F}. 
In a different direction, Codenotti--Riccardi--Venturello studied the number of edges of symmetric edge polytopes 
and characterized the graphs attaining the sharp lower bound~\cite{CRV}. 
Moreover, the facets of symmetric edge polytopes admit a rich graph-theoretic description. 
In particular, their supporting functions can be described in terms of integer-valued functions on the vertex set of the underlying graph, 
and the corresponding facet subgraphs have been studied in detail~\cite{HJM,CDK}. 

A graph $G$ is called a \emph{complete multipartite graph} if its vertex set can be partitioned into nonempty independent sets 
\[ V(G)=V_1\sqcup\cdots\sqcup V_r \]
such that two distinct vertices are adjacent if and only if they belong to different parts. We write $G=K_{a_1,\ldots,a_r}$, where $a_i=|V_i|$. 

Our main result is the following classification.
\begin{theorem}\label{thm:main}
Let $G$ be a connected graph with $n\ge 3$ vertices. Then $P_G^{\pm}$ is Minkowski decomposable if and only if 
\[G\cong K_n,\qquad G\cong K_{2,n-2},\qquad\text{or}\qquad G\cong K_{1,1,n-2}.\]
\end{theorem}
Note that $P_G^\pm$ is Minkowski indecomposable if $G$ is not $2$-connected by Corollary~\ref{cor:2-conn}, 
so we focus on the case of $2$-connected graphs for the proof of the ``only if'' part of Theorem~\ref{thm:main}.

For the three exceptional families, we give explicit Minkowski decompositions in Section~\ref{sec:exceptional}. 
For all other $2$-connected graphs, the proof uses McMullen's criterion (Theorem~\ref{thm:mcmullen}). 
The main combinatorial discussion is a series of transitions for triangular faces of $P_G^{\pm}$. 
Roughly speaking, unless $G$ is one of the three exceptional families, there is a strongly connected family of triangular faces which touches every facet of $P_G^{\pm}$. 
Since triangles are Minkowski indecomposable, McMullen's criterion implies the Minkowski indecomposability of $P_G^{\pm}$. 

\begin{remark}\label{rem:CRV-connection}
The three exceptional graphs in Theorem~\ref{thm:main} coincide exactly with the graphs appearing in the main result of Codenotti--Riccardi--Venturello~\cite[Corollary~4.3]{CRV}. 
They prove that, for a connected graph $G$, the sharp lower bound 
\begin{align}\label{eq:lowerbound} f_1(P_G^{\pm})\ge |E(G)|(2|V(G)|-5)-|E_3(G)| \end{align}
is attained if and only if $G\cong K_n$, $G\cong K_{2,n-2}$, or $G\cong K_{1,1,n-2}$, 
where $f_1(P)$ denotes the number of edges of a polytope $P$ and $E_3(G):=\{e \in E(G) : e \text{ belongs to a $3$-cycle in }G\}$. 
Namely, $P_G^\pm$ is Minkowski decomposable if and only if $G$ satisfies the equality of \eqref{eq:lowerbound}. 
This coincidence appears to reflect a common underlying mechanism. 
The counting of the edges is governed by the one-dimensional face structure of $P_G^\pm$: 
two vertices of $P_G^\pm$ form an edge precisely when the corresponding directed edges of $G$ do not lie on a common directed cycle of length $3$ or $4$. 
In our paper, the triangular faces are controlled by the same pairwise condition, 
together with the additional exclusion of directed cycles of length $5$ or $6$ containing all three directed edges. 
Hence their result concerns the number of ``safe pairs'', while our proof uses the transition connectedness of the corresponding ``safe triples''. 
\end{remark}

\medskip

We describe a brief structure of this paper. 
In Section~\ref{sec:triangles}, we characterize the triangular faces of $P_G^{\pm}$ in terms of directed edges of $G$. 
In Section~\ref{sec:criterion}, we recall McMullen's criterion for Minkowski indecomposability and introduce the strongly connected families of triangular faces used in the proof. 
In Section~\ref{sec:transition}, we prove the transition classification for $2$-connected graphs. 
In Section~\ref{sec:exceptional}, we give explicit Minkowski decompositions for the three exceptional families. 
Finally, in Section~\ref{sec:proof-main}, we give a proof of Theorem~\ref{thm:main}. 

\medskip

\section*{Acknowledgments}
The first author was partially supported by JSPS KAKENHI Grant Numbers JP24K00521 and JP24K00534. 
The second author was supported by JSPS KAKENHI Grant Number JP26K17026.

\bigskip
\section{Triangular faces of symmetric edge polytopes}\label{sec:triangles}

In this section, we provide a characterization of the triangular faces of $P_G^{\pm}$ in terms of directed edges of $G$.
We use the following elementary observation, which is useful for analyzing not only $2$-dimensional  but also higher-dimensional faces of $P_G^\pm$.  
\begin{lemma}\label{lem:cycle-face-obstruction}
Let $F$ be a proper face of $P_G^{\pm} \subset \RR^n$, and let $C=(v_1,v_2,\ldots,v_\ell,v_1)$ be a directed cycle in $G$. 
Let \[E_F(C)=\{v_i v_{i+1}:\rho(v_i v_{i+1})\in F\},\] where $v_{\ell+1}=v_1$. Then \[|E_F(C)|\le \frac{\ell}{2}.\] 
Moreover, if the equality holds, then for every directed edge $v_i v_{i+1}$ of $C$ with $v_i v_{i+1}\notin E_F(C)$, one has $\rho(v_{i+1}v_i)\in F$. 
\end{lemma}
\begin{proof}
Let $h\in (\RR^n)^*$ be a linear functional defining $F$ such that $h(x)=1$ for $x\in F$ and $h(x)<1$ for $x\in P_G^{\pm}\setminus F$. 
Note that we can find such $h$ since $P_G^\pm$ is a polytope containing the origin in its relative interior. 
Write $h_i=h(\eb_i)$ for $i=1,\ldots,n$. For any edge $\{u,v\}$ in $G$, the inequalities $h(\rho(uv))\le 1$ and $h(\rho(vu)) \le 1$ hold. 
Hence, we have $|h_u-h_v|\le 1$ for every $\{u,v\}\in E(G)$. 

Let $k=|E_F(C)|$. If $v_i v_{i+1}\in E_F(C)$, then $h_{v_i}-h_{v_{i+1}}=1$. 
For every directed edge $v_i v_{i+1}$ of $C$, we also have $h_{v_i}-h_{v_{i+1}}\ge -1$. By summing along the directed edges of $C$, we obtain that 
\[0=\sum_{i=1}^{\ell}(h_{v_i}-h_{v_{i+1}})\ge k-(\ell-k)=2k-\ell.\]
Thus, $k\le \ell/2$. If the equality holds, then all of the remaining terms must be $-1$. 
Hence, $h_{v_{i+1}}-h_{v_i}=1$ whenever $v_i v_{i+1}\notin E_F(C)$, which means $\rho(v_{i+1}v_i)\in F$. 
\end{proof}

We introduce three conditions for a set $T$ of three directed edges of $G$:
\begin{enumerate}[label=\textup{(A\arabic*)}]
\item $T$ contains no opposite pair $uv,vu$.
\item No two directed edges in $T$ are contained in a common directed cycle of length $3$ or $4$ in $G$.
\item The three directed edges in $T$ are not contained in a common directed cycle of length $5$ or $6$ in $G$. 
\end{enumerate}
For a set $T$ of three directed edges, put \[  \Delta_T=\conv\{\rho(e):e\in T\}.\]

\begin{proposition}\label{prop:triangles}
Let $T$ be a set of three directed edges of $G$.  Then $\Delta_T$ is a triangular face of $P_G^{\pm}$ 
if and only if $T$ satisfies \textup{(A1)}, \textup{(A2)}, and \textup{(A3)}. 
\end{proposition}
\begin{proof}
{\bf (Only if)}: Assume that $\Delta_T$ is a triangular face. 
Since $P_G^{\pm}$ is centrally symmetric, no proper face contains a pair of antipodal vertices. Hence, $T$ contains no opposite pair, i.e., satisfies (A1). 
Moreover, by the characterization of the edges of symmetric edge polytopes \cite[Lemma~3.2]{CRV}, 
two vertices $\rho(e)$ and $\rho(f)$ span an edge of $P_G^{\pm}$ if and only if the directed edges $e$ and $f$ are not contained in a common directed cycle of length $3$ or $4$. 
Since every pair of vertices of a triangular face spans an edge, (A2) follows. 
Finally, if the three directed edges in $T$ were contained in a common directed cycle $C$ of length $5$ or $6$, 
then Lemma~\ref{lem:cycle-face-obstruction} gives \[3\le |E_{\Delta_T}(C)|\le \frac{\ell(C)}{2}.\]
This is impossible for $\ell(C)=5$. If $\ell(C)=6$, the equality holds and Lemma~\ref{lem:cycle-face-obstruction} implies that 
$\Delta_T$ contains additional vertices, contradicting that it is a triangle. Thus, (A3) holds. 

\bigskip

\noindent
{\bf (If)}: Assume that $T$ satisfies (A1), (A2), and (A3). Our goal is to construct a linear functional whose maximizers on $P_G^{\pm}$ are precisely $\{\rho(e):e\in T\}$. 
We use the standard feasibility criterion for systems of difference constraints: a system $z_j-z_i\le c_{ij}$ has a solution if and only if 
the associated weighted directed graph has no directed cycle of negative total weight; see, e.g., \cite[Theorem~24.9]{CLRS}. 

Choose $0<\varepsilon<1/n$ and consider the following system of constraints: 
\[h_x-h_y=1 \quad \text{for }xy\in T, \;\;\text{ and }\;\; h_x-h_y\le 1-\varepsilon \quad \text{for }\{x,y\} \in E(G) \text{ with }xy \not\in T.\]
Equivalently, for each directed edge $e=uv$, we impose $h_v-h_u \le c(e)$ with 
\[c(e)=
\begin{cases}
-1, & e \in T,\\ 
1, & \overline e\in T,\\
1-\varepsilon, & e\notin T \text{ and } \overline e\notin T,
\end{cases}
\]
where $\overline e=vu$ is the opposite directed edge. 
It remains to show that there is no negative directed cycle by \cite[Theorem~24.9]{CLRS}. 
Let $C$ be a directed cycle in $G$, let $\ell=\ell(C)$, let $k=|T\cap E(C)|$, and 
let $s$ be the number of edges $e\in E(C)$ such that neither $e$ nor $\overline e$ belongs to $T$. 
The total weight of $C$ is \[-k+(\ell-k-s)+(1-\varepsilon)s=\ell-2k-\varepsilon s.\]

\begin{itemize}
\item If $\ell=2$, then $C$ consists of opposite directed edges.
By (A1), $k\le1$. If $k=0$, then the total weight is $2-2\varepsilon>0$, while if $k=1$, then $s=0$, so the total weight is $0$.
\item If $\ell\ge3$, then
\[
k\le
\begin{cases}
1,&\ell=3 \text{ or } 4,\\
2,&\ell=5 \text{ or } 6,\\
3,&\ell\ge7,
\end{cases}
\]
by (A2), (A3), and $k\le|T|$, respectively.
Hence $\ell-2k\ge1$.
Since $s\le\ell\le n$ and $\varepsilon<1/n$, we obtain
\[
\ell-2k-\varepsilon s>0.
\]
\end{itemize}
These imply that there is no negative directed cycle in $G$. 

Thus, the system has a solution $h$. For this solution, $h(\rho(e))=1$ for $e\in T$ and $h(\rho(f))<1$ for every directed edge $f\notin T$. 
Therefore, the face defined by $h$ is $\Delta_T$, which has precisely the three vertices $\rho(e)$, $e\in T$, as desired.
\end{proof}

\bigskip
\section{McMullen's criterion}\label{sec:criterion}

In this section, we recall McMullen's criterion for Minkowski indecomposability (Theorem~\ref{thm:mcmullen}) and give its corollaries. 

\begin{definition}\label{def:Mc}
Let $P$ be a convex polytope. 
\begin{itemize}
\item A sequence of faces $F_0,F_1,\ldots,F_k$ of $P$ is a \emph{strong chain} if $\dim(F_{j-1}\cap F_j)\ge 1$ for $j=1,\ldots,k$. 
\item A family $\Fc$ of faces of $P$ is \emph{strongly connected} if, for any $F,G\in\Fc$, there is a strong chain $F=F_0,F_1,\ldots,F_k=G$ with $F_j\in\Fc$ for all $j$. 
\item A family $\Fc$ of faces of $P$ \emph{touches} a face $H$ of $P$ if $\bigl(\bigcup_{F\in\Fc}F\bigr)\cap H\ne\emptyset$. 
\end{itemize}
\end{definition}

\begin{theorem}[{\cite[Theorem~2]{McM}}]\label{thm:mcmullen}
If a convex polytope $P$ has a strongly connected family of indecomposable faces which touches every facet of $P$, then $P$ is Minkowski indecomposable. 
\end{theorem}

Note that triangles are Minkowski indecomposable, so a strongly connected family of triangular faces touching every facet is enough to prove Minkowski indecomposability. 
\begin{corollary}\label{cor:2-conn}
Let $G$ be a connected graph with $n\ge 4$ vertices. If $G$ is not $2$-connected, then $P_G^\pm$ is Minkowski indecomposable. 
\end{corollary}
\begin{proof}
Let $B_1,\ldots,B_s$ be the $2$-connected components of $G$. Since $G$ is connected but not $2$-connected, we have $s\ge 2$. 
We regard each $P_{B_i}^{\pm}$ as a polytope in $\RR^n$. 

Here, we recall that for two polytopes $P_1 \subset \RR^{n_1}$ and $P_2 \subset \RR^{n_2}$ containing the origin in 
their interiors, the \emph{free sum} and the \emph{join} of $P_1$ and $P_2$ are defined, respectively, by
\[P_1 \oplus P_2 = \conv(\{(\alpha_1,{\bf 0}_2) : \alpha_1 \in P_1\} \cup \{({\bf 0}_1, \alpha_2) : \alpha_2 \in P_2\}) \subset \RR^{n_1+n_2}\] 
and 
\[P_1 * P_2 = \conv(\{(\alpha_1,{\bf 0}_2,0) : \alpha_1 \in P_1\} \cup \{({\bf 0}_1, \alpha_2,1) : \alpha_2 \in P_2\}) \subset \RR^{n_1+n_2+1}.\] 
Note that $2$-connected components of $G$ correspond to free summands of $P_G^\pm$, see \cite{MHNOH}.

We emphasize the distinction between a free sum and a join. 
Although $P_G^\pm$ itself is the free sum of symmetric edge polytopes of the $2$-connected components, 
we see that the facets of a free sum are affinely isomorphic to the joins of facets of the summands. 
Hence, by \cite[Theorem~3]{McM}, it follows that every facet of $P_G^\pm$ is Minkowski indecomposable. 

Finally, in general, the family of all facets of a polytope of dimension at least $3$ is strongly connected and trivially touches every facet. 
Therefore, we obtain Minkowski indecomposability by Theorem~\ref{thm:mcmullen}, as required.  
\end{proof}

\medskip

Let $\Tc(G)$ be the set of all three-element sets of directed edges satisfying (A1), (A2), and (A3). 
For $\Fc\subset \Tc(G)$, let \[\rho(\Fc)=\{\Delta_T:T\in\Fc\}.\]
Namely, $\rho(\Fc)$ corresponds to a collection of triangular faces of $P_G^\pm$. 
We say that two members of $\Tc(G)$ are \emph{adjacent} if they share exactly two directed edges. 
A subfamily $\Fc\subset \Tc(G)$ is \emph{connected} if its adjacency graph is connected. 
By Proposition~\ref{prop:triangles}, the connectedness of $\Fc$ implies the strong connectedness of $\rho(\Fc)$ in the sense of Definition~\ref{def:Mc}. 
\begin{proposition}\label{prop:cycle}
Let $G$ be a cycle of length at least $5$. Then $\rho(\Tc(G))$ is a strongly connected family of triangular faces which touches every facet of $P_G^{\pm}$. 
In particular, $P_G^{\pm}$ is Minkowski indecomposable. 
\end{proposition}
\begin{proof}
Let $G$ be a cycle of length $n \ge 5$ and write its edges as $e_i=\{i,i+1\}$ for $i=1,\ldots,n$, where $e_n=\{1,n\}$. 
Let $e_i^+$ and $e_i^-$ be the two directed edges corresponding to the two cyclic orientations of $G$. 
Since $G$ has no cycles of length $3$ or $4$, the condition (A2) imposes no restriction. 
The only directed cycles of length at least $5$ are the two cyclic orientations of the whole cycle. 
Hence, if $n \ge 7$, then a set of three directed edges belongs to $\Tc(G)$ if and only if it contains no opposite pair. 
If $n=5$ or $6$, then it belongs to $\Tc(G)$ if and only if it contains no opposite pair and is not contained in one of the two cyclic orientations of $G$. 

We show that $\Tc(G)$ is connected. 
Each $T\in \Tc(G)$ can be written as $T=\{e_i^{\sigma_i}:i\in I\}$, where $I\subset [n]$, $|I|=3$, and $\sigma_i\in\{+,-\}$. 
It is well known that the graph of $3$-element subsets of $[n]$, where two subsets are adjacent if they differ in one element, is connected
(since this is nothing but a Johnson graph $J(n,3)$, which is known to be connected).  
Hence, we see the connectedness of the members of $\Tc(G)$ by ignoring the signs $\sigma_i$. 
Hence, it remains to consider the possible sign patterns for a fixed $I$.
If $n \ge 7$, all eight sign patterns are allowed, and the corresponding adjacency graph is the $3$-cube. If $n=5$ or $6$, the number of the allowed patterns is six, namely, all except the two patterns in which all signs are equal, and the corresponding adjacency graph is a $6$-cycle. In either case, the adjacency graph is connected.
Therefore, $\Tc(G)$ is connected. 

It remains to show that $\rho(\Tc(G))$ touches every facet. Let $F$ be a facet of $P_G^{\pm}$, and choose a vertex $\rho(e)$ of $F$. 
Then we can extend the directed edge $e$ to some $T\in\Tc(G)$ by choosing two further edges of the cycle together with suitable orientations so that no opposite pair occurs. 
This implies that $F \cap \Delta_T\ne\emptyset$, i.e., $\rho(\Tc(G))$ touches every facet of $P_G^\pm$. 

Therefore, the conclusion follows from Theorem~\ref{thm:mcmullen}, as desired. 
\end{proof}
\begin{remark}
Note that we know the Minkowski indecomposability of symmetric edge polytopes of odd cycles of length at least $5$ more directly. 
In fact, by \cite[Corollary 2.3]{H}, $P_G^\pm$ is simplicial if and only if $G$ contains no even cycle. 
In general every simplicial polytope of dimension at least $3$ is Minkowski indecomposable. 
(This follows by considering all $2$-dimensional faces and applying Theorem~\ref{thm:mcmullen}.) 
\end{remark}

\bigskip
\section{Transition of triples of directed edges}\label{sec:transition}

For $u\in V(G)$, let $N_G(u)$ be the set of neighbours of $u$ and
$N_G[u]=N_G(u)\cup\{u\}$.  A vertex $u$ of degree at least $3$ is called
\emph{good} if there exist distinct $a,b,c\in N_G(u)$ such that
\[\{au,bu,cu\}\leadsto \{ua,ub,uc\},\]
where $\leadsto$ denotes the existence of a sequence of adjacent triples in $\Tc(G)$. 
A triple of directed edges will be called \emph{valid} if it belongs to $\Tc(G)$. 

\begin{proposition}\label{prop:transition-classification}
Let $G$ be a 2-connected graph with $n\ge 5$ vertices.  Suppose that $G$ has
at least two vertices of degree at least $3$.  If $G$ has no good vertex, then
\[G\cong K_n,\qquad G\cong K_{2,n-2},\qquad\text{or}\qquad G\cong K_{1,1,n-2}.\]
\end{proposition}

We first record the local switches used in the proof of Proposition~\ref{prop:transition-classification}. 
In each case, the displayed sequence consists of members of $\Tc(G)$. We may verify (A1), (A2) and (A3) in a direct way. 
\begin{lemma}[Local switches]\label{lem:local-switches}
Let $u$ be a vertex of degree at least $3$. Then $u$ is good in each of the following four cases. See Figure~\ref{fig:local}. 
\begin{enumerate}[label=\textup{(\alph*)}]
\item There exist $q,r\in V(G)\setminus N_G[u]$ and $c\in N_G(u)$ such that $qr\in E(G)$ and $qc\in E(G)$.
\item There exist distinct non-adjacent vertices $q,r\in V(G)\setminus N_G[u]$ and distinct vertices $a,c\in N_G(u)$ such that $qa\in E(G)$ and $rc\in E(G)$.
\item There exists $w\notin N_G[u]$ such that $N_G(w)\cap N_G(u)$ has at least two vertices, and some vertex of $N_G(w)\cap N_G(u)$ has a neighbour in $N_G(u)$.
\item There exist distinct vertices $p,q,a,b\in N_G(u)$ such that $pq\notin E(G)$, $pa\in E(G)$, and $qb\in E(G)$. 
\end{enumerate}
\end{lemma}
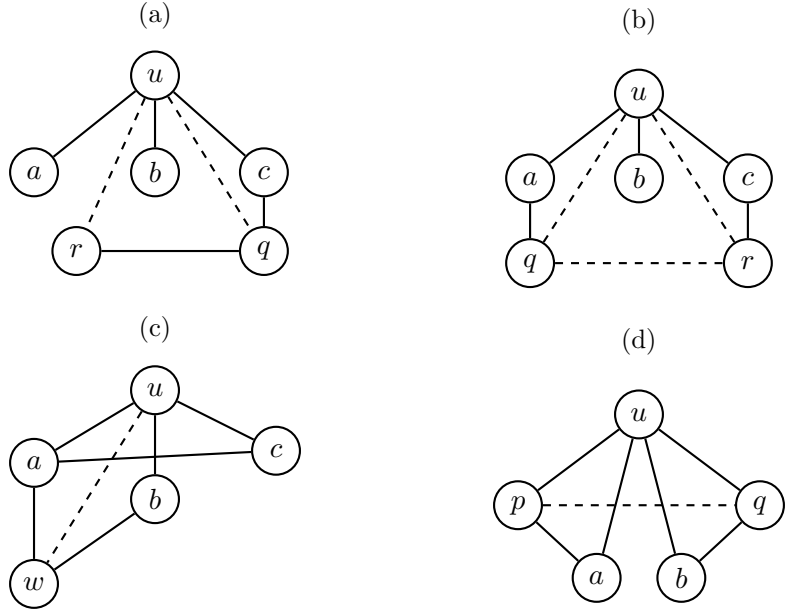
\begin{figure}[htbp]
\centering
\begin{tikzpicture}[scale=0.8,  vertex/.style={circle, draw, thick, inner sep=1.2pt, minimum size=18pt},  edge/.style={thick},  nonedge/.style={thick, dashed},  lab/.style={font=\small}]
%%%%%%%%%%%%%%%%%%%%%%%%%%%%%%%%%%%%%%%%%%%%%%%%%%
% (a)
%%%%%%%%%%%%%%%%%%%%%%%%%%%%%%%%%%%%%%%%%%%%%%%%%%
\begin{scope}[xshift=3.2cm,yshift=0cm]
  \node[lab] at (0,2.9) {(a)};
  \node[vertex] (u) at (0,1.9) {$u$};
  \node[vertex] (a) at (-2.0,0.3) {$a$};
  \node[vertex] (b) at (0,0.3) {$b$};
  \node[vertex] (c) at (1.8,0.3) {$c$};
  \node[vertex] (q) at (1.8,-1.0) {$q$};
  \node[vertex] (r) at (-1.3,-1.0) {$r$};

  \draw[edge] (u)--(a);
  \draw[edge] (u)--(b);
  \draw[edge] (u)--(c);
  \draw[edge] (q)--(c);
  \draw[edge] (q)--(r);

  \draw[nonedge] (u)--(q);
  \draw[nonedge] (u)--(r);
\end{scope}
%%%%%%%%%%%%%%%%%%%%%%%%%%%%%%%%%%%%%%%%%%%%%%%%%%
% (b)
%%%%%%%%%%%%%%%%%%%%%%%%%%%%%%%%%%%%%%%%%%%%%%%%%%
\begin{scope}[xshift=11.2cm,yshift=0cm]
  \node[lab] at (0,2.8) {(b)};
  \node[vertex] (u) at (0,1.6) {$u$};
  \node[vertex] (a) at (-1.8,0.2) {$a$};
  \node[vertex] (b) at (0,0.2) {$b$};
  \node[vertex] (c) at (1.8,0.2) {$c$};
  \node[vertex] (q) at (-1.8,-1.2) {$q$};
  \node[vertex] (r) at (1.8,-1.2) {$r$};

  \draw[edge] (u)--(a);
  \draw[edge] (u)--(b);
  \draw[edge] (u)--(c);
  \draw[edge] (q)--(a);
  \draw[edge] (r)--(c);

  \draw[nonedge] (q)--(r);
  \draw[nonedge] (u)--(q);
  \draw[nonedge] (u)--(r);
\end{scope}
%%%%%%%%%%%%%%%%%%%%%%%%%%%%%%%%%%%%%%%%%%%%%%%%%%
% (c)
%%%%%%%%%%%%%%%%%%%%%%%%%%%%%%%%%%%%%%%%%%%%%%%%%%
\begin{scope}[xshift=3.2cm,yshift=-5.3cm]
  \node[lab] at (0,3.0) {(c)};
  \node[vertex] (u) at (0,2.0) {$u$};
  \node[vertex] (a) at (-2.0,0.8) {$a$};
  \node[vertex] (b) at (0,0.2) {$b$};
  \node[vertex] (c) at (2.0,1.0) {$c$};
  \node[vertex] (w) at (-2.0,-1.2) {$w$};

  \draw[edge] (u)--(a);
  \draw[edge] (u)--(b);
  \draw[edge] (u)--(c);
  \draw[edge] (w)--(a);
  \draw[edge] (w)--(b);
  \draw[edge] (a)--(c);

  \draw[nonedge] (u)--(w);

% \node[lab] at (0,-2.0) {$a,b\in N_G(w)\cap N_G(u),\ ac\in E(G)$};
\end{scope}
%%%%%%%%%%%%%%%%%%%%%%%%%%%%%%%%%%%%%%%%%%%%%%%%%%
% (d)
%%%%%%%%%%%%%%%%%%%%%%%%%%%%%%%%%%%%%%%%%%%%%%%%%%
\begin{scope}[xshift=11.2cm,yshift=-5.3cm]
  \node[lab] at (0,2.8) {(d)};
  \node[vertex] (u) at (0,1.6) {$u$};
  \node[vertex] (p) at (-2.0,0.1) {$p$};
  \node[vertex] (a) at (-0.7,-1.1) {$a$};
  \node[vertex] (b) at (0.7,-1.1) {$b$};
  \node[vertex] (q) at (2.0,0.1) {$q$};

  \draw[edge] (u)--(p);
  \draw[edge] (u)--(a);
  \draw[edge] (u)--(b);
  \draw[edge] (u)--(q);
  \draw[edge] (p)--(a);
  \draw[edge] (q)--(b);

  \draw[nonedge] (p)--(q);
\end{scope}
\end{tikzpicture}
\caption{The four local configurations in Lemma~\ref{lem:local-switches}.
Solid edges indicate edges of $G$, and dashed edges indicate non-edges.}\label{fig:local}
\end{figure}
\begin{proof}
{\bf (a)}: Choose distinct $a,b\in N_G(u)\setminus\{c\}$. Since $q,r\notin N_G[u]$, we have $uq,ur\notin E(G)$.
Consider the sequence  
\[\{au,bu,cu\}\to \{au,bu,qr\}\to \{au,qc,qr\}\to \{uc,qc,qr\}\to \{ua,uc,qr\}\to \{ua,ub,uc\}.\]
Each step changes exactly one directed edge.

We check that all triples are valid. The first and last triples are stars, and hence they are valid. 
For the intermediate triples, it is enough to check (A2), since (A1) is immediate and (A3) is excluded because each intermediate triple contains two directed edges with a common initial vertex or a common terminal vertex. 
For the triple $\{au,bu,qr\}$, the pair $au,bu$ has a common terminal vertex $u$. 
The pairs $bu, qr$ and $au, qr$ would each require the edge $uq$ in a directed $4$-cycle, which is impossible.  
For the triple $\{au,qc,qr\}$, the pair $qc,qr$ has a common initial vertex $q$. 
The pairs $au,qc$ and $au,qr$ would again each require the edge $uq$ in a directed $4$-cycle, which is impossible.  
For the triple $\{uc,qc,qr\}$, the pair $qc,qr$ has a common initial vertex $q$, and the pair $uc,qc$ has a common terminal vertex $c$. 
The remaining pair $uc,qr$ would require the edge $ur$ in a directed $4$-cycle, which is impossible.  
For the triple $\{ua,uc,qr\}$, the pair $ua,uc$ has a common initial vertex $u$.
The pairs $ua,qr$ and $uc,qr$ would again each require the edge $ur$ in a directed $4$-cycle, which is impossible.  
Thus, every triple in the displayed sequence satisfies (A1), (A2) and (A3). Therefore, this is a valid transition, and $u$ is good. 

The validity of the transitions in cases {\rm (b)}--{\rm (d)} can be verified in the same way as in {\rm (a)}.

\noindent
{\bf (b)}: Choose $b\in N_G(u)\setminus\{a,c\}$.  A valid transition is given as follows: 
\[\{au,bu,cu\}\to \{qa,bu,cu\}\to \{qa,cu,cr\}\to \{ua,qa,cr\}\to \{ua,ub,cr\}\to \{ua,ub,uc\}.\]

\noindent
{\bf (c)}: 
Choose $a\in N_G(w)\cap N_G(u)$ and $c\in N_G(u)$ with $ac\in E(G)$. 
Choose $b\in N_G(u)\setminus\{a,c\}$ so that \[wb\in E(G)\quad\text{or}\quad wc\in E(G).\]
We give an explicit valid transition. There are three cases.

If $wc\in E(G)$ and $wb\notin E(G)$, use
\[\{au,bu,cu\}\to\{bu,cu,wa\}\to\{bu,wa,wc\}\to\{ub,wa,wc\}\to\{ua,ub,wa\}\to\{ua,ub,uc\}.\]
If $wb\in E(G)$ and $wc\notin E(G)$, use
\[\{au,bu,cu\}\to\{bu,cu,wa\}\to\{cu,wa,wb\}\to\{uc,wa,wb\}\to\{ua,uc,wa\}\to\{ua,ub,uc\}.\]
Finally, if $wb,wc\in E(G)$, use
\begin{align*}\begin{aligned}
\{au,bu,cu\}&\to\{bu,cu,wa\}\to\{ca,cu,wa\}\to\{ca,ua,wa\}\\
&\to\{ca,cw,ua\}\to\{cw,ua,ub\}\to\{ua,ub,uc\}.\end{aligned}\end{align*}

\noindent
{\bf (d)}: A valid transition is as follows: 
\[\begin{aligned}
\{au,bu,pu\}&\to \{bu,bq,pu\}\to \{bq,pu,pa\}\to \{uq,bq,pa\}\\
&\to \{ua,uq,pa\}\to \{ua,ub,uq\}\to \{ua,ub,up\}.\end{aligned}\]
\end{proof}

\begin{proof}[Proof of Proposition~\ref{prop:transition-classification}]
Choose a vertex $u$ of degree at least $3$. Then $u$ is not good. 
We divide our discussion into two cases: $V(G)\setminus N_G[u]\ne\emptyset$ or $V(G) = N_G[u]$. 

\medskip

\noindent
($V(G)\setminus N_G[u]\ne\emptyset$): In this case, we show that $V(G)\setminus N_G[u]$ is independent. 
Suppose for contradiction that the subgraph induced by $V(G)\setminus N_G[u]$ contains an edge. 
Let $C$ be a connected component of this induced subgraph containing an edge. Since $G$ is connected, $C$ has a neighbour in $N_G(u)$. 
Choose a vertex $q\in V(C)$ adjacent to some $c\in N_G(u)$. Since $C$ contains an edge and is connected, we may choose such $q$ with a neighbour $r\in V(C)$. 
Then \[q,r\in V(G)\setminus N_G[u],\qquad qr\in E(G),\qquad qc\in E(G).\]
By Lemma~\ref{lem:local-switches} (a), the vertex $u$ is good, a contradiction. Hence $V(G)\setminus N_G[u]$ is independent.

Now, let $w\in V(G)\setminus N_G[u]$. Since $V(G)\setminus N_G[u]$ is independent, all neighbours of $w$ lie in $N_G(u)$. 
Moreover, $w$ must have at least two neighbours in $N_G(u)$; otherwise, its unique neighbour in $N_G(u)$ would be a cut vertex separating $w$ from the remaining part of $G$, 
a contradiction to the $2$-connectedness of $G$. 

We next show that $V(G)\setminus N_G[u]$ consists of a single vertex. 
Suppose for contradiction that there are two distinct vertices \[q,r\in V(G)\setminus N_G[u].\] 
Since $V(G)\setminus N_G[u]$ is independent, we have $qr\notin E(G)$ and both $q$ and $r$ have at least two neighbours in $N_G(u)$, 
so we can choose distinct vertices $a,c\in N_G(u)$ such that \[qa\in E(G),\qquad rc\in E(G).\]
Then Lemma~\ref{lem:local-switches} (b) implies that $u$ is good, again a contradiction. Therefore, \[V(G)\setminus N_G[u]=\{w\}\] 
for some vertex $w$. 

We claim that every vertex in $N_G(w)\cap N_G(u)$ has no neighbour inside $N_G(u)$. 
Since $w\in V(G)\setminus N_G[u]$ and $G$ is $2$-connected, the set $N_G(w)\cap N_G(u)$ has at least two vertices. 
If some $a\in N_G(w)\cap N_G(u)$ had a neighbour in $N_G(u)$, then Lemma~\ref{lem:local-switches} (c) would imply that $u$ is good, a contradiction. 
Hence, every vertex in $N_G(w)\cap N_G(u)$ has no neighbour inside $N_G(u)$. 

We show that \[N_G(u) \subset N_G(w).\] 
On the contrary, if there exists a vertex $z\in N_G(u)\setminus N_G(w)$, then the set
\[\{w\}\cup \bigl(N_G(w)\cap N_G(u)\bigr)\] has no edge to \[N_G(u)\setminus N_G(w).\]
Indeed, $w$ has no neighbour in $N_G(u)\setminus N_G(w)$ by definition, and the vertices in $N_G(w)\cap N_G(u)$ have no neighbours inside $N_G(u)$. 
Hence, $u$ is a cut vertex, a contradiction. Therefore, \[N_G(w)\cap N_G(u)=N_G(u).\]
It follows that no two vertices in $N_G(u)$ are adjacent. Thus, $u$ and $w$ are both adjacent to every vertex in $N_G(u)$. Moreover, $uw\notin E(G)$, 
and there are no edges among the vertices of $N_G(u)$. Hence, \[G\cong K_{2,n-2}.\]

\medskip

\noindent
($V(G) = N_G[u]$): In this case, we have $V(G)=\{u\}\sqcup N_G(u)$. 
Since $G$ is $2$-connected, the graph $G - u$ is connected. 
If every two distinct vertices in $N_G(u)$ are adjacent, then $G\cong K_n$. 

Assume otherwise. Choose a non-edge $pq\notin E(G)$ with $p,q\in N_G(u)$. 
Then both $p$ and $q$ have neighbours in $N_G(u)$ by the connectedness of $G - u$. 
If there are \[a\in N_G(u)\quad\text{with }pa\in E(G),\qquad b\in N_G(u)\quad\text{with }qb\in E(G) \;\;\text{and}\;\; a\ne b, \]
then Lemma~\ref{lem:local-switches} (d) would imply that $u$ is good, a contradiction. 
Therefore, all neighbours of $p$ inside $N_G(u)$ and all neighbours of $q$ inside $N_G(u)$ coincide with one single vertex. 
Thus, there exists $w\in N_G(u)$ such that \[N_G(p)\cap N_G(u)=N_G(q)\cap N_G(u)=\{w\}.\]

We show that the graph induced on $N_G(u)$ is a star with center $w$. Let \[s\in N_G(u)\setminus\{w\}.\]
If $s=p$, then \[N_G(s)\cap N_G(u)=\{w\}.\]
If $s\ne p,w$, then $ps\notin E(G)$, because $p$ has no neighbour in $N_G(u)$ other than $w$. 
By using the non-edge $ps$ together with $N_G(p)\cap N_G(u)=\{w\}$, 
we obtain \[N_G(s)\cap N_G(u)=\{w\}.\]
Thus, every vertex in $N_G(u)\setminus\{w\}$ is adjacent only to $w$ within $N_G(u)$. 
Hence, the graph induced on $N_G(u)$ is a star with center $w$. 

Therefore, $u$ and $w$ are adjacent, and all remaining $n-2$ vertices are adjacent exactly to $u$ and $w$. 
Hence, \[G\cong K_{1,1,n-2}.\]
This completes the proof.
\end{proof}

\bigskip
\section{The exceptional families}\label{sec:exceptional}

In this section, we record explicit nontrivial Minkowski decompositions for the three exceptional families $K_n$, $K_{2,n-2}$ and $K_{1,1,n-2}$. 

\noindent
($K_n$): Let \[\Delta_{n-1}=\conv\{\eb_1,\ldots,\eb_n\}\] be the standard simplex. 
Then \[P_{K_n}^\pm=\Delta_{n-1}+(-\Delta_{n-1}). \]

\noindent
($K_{2,n-2}$): Let $G=K_{2,n-2}$ with parts $\{1,2\}$ and $\{3,\ldots,n\}$. 
Then \[P_{K_{2,n-2}}^{\pm}=\conv\{\eb_1-\eb_i,\ \eb_i-\eb_2:3\le i\le n\}+\conv\{{\bf 0},\eb_2-\eb_1\}.\]

\noindent
($K_{1,1,n-2}$): Let $G=K_{1,1,n-2}$ with parts $\{1\}$, $\{2\}$ and $\{3,\ldots,n\}$. 
Then \[P_{K_{1,1,n-2}}^{\pm}=\conv\{{\bf 0},\eb_1-\eb_2,\eb_1-\eb_i,\eb_i-\eb_2:3\le i\le n\}+\conv\{{\bf 0},\eb_2-\eb_1\}. \]

\medskip

In all three cases, the displayed decomposition is nontrivial. 
Therefore, each graph in the three exceptional families gives a Minkowski decomposable symmetric edge polytope.

\bigskip
\section{Proof of the main theorem for $2$-connected graphs}\label{sec:proof-main}

Now, we are ready to give a proof of the main theorem. 
When $n=3$, the connected graphs are $K_3$ and $K_{2,1}$, both of which belong to the exceptional families. Hence, we may assume that $n\ge4$.
If $G$ is connected but not $2$-connected, then $P_G^\pm$ is Minkowski indecomposable by Corollary~\ref{cor:2-conn}. 
Hence, in what follows, let $G$ be $2$-connected. 
Note that the only $2$-connected graphs on four vertices are $K_4$, $K_{2,2}$, and $K_{1,1,2}$, all of which belong to the exceptional families.
Therefore, in what follows, we may also assume that $n\ge 5$.

\begin{proof}[Proof of Theorem~\ref{thm:main}]
{\bf (Only if)}: 
Suppose that $G$ is 2-connected and is not isomorphic to any of $K_n$, $K_{2,n-2}$, and $K_{1,1,n-2}$.
If $G$ is a cycle, then Proposition~\ref{prop:cycle} shows that $P_G^{\pm}$ is Minkowski indecomposable. 
Hence, assume that $G$ is not a cycle. Then $G$ has at least two vertices of degree at least $3$. 
By Proposition~\ref{prop:transition-classification}, $G$ has a good vertex $u$. 
Choose distinct $a,b,c\in N_G(u)$ and a transition \[\{au,bu,cu\}=T_0\to T_1\to\cdots\to T_q=\{ua,ub,uc\}.\]
Let
\[\Fc^-:=\{\{xu,yu,zu\}:x,y,z\in N_G(u) \text{ are distinct}\} \text{ and}\]
\[\Fc^+:=\{\{ux,uy,uz\}:x,y,z\in N_G(u) \text{ are distinct}\}.\]
Set
\[\Fc:=\{T_0,\ldots,T_q\}\cup\Fc^-\cup\Fc^+.\]
Both $\Fc^-$ and $\Fc^+$ are connected in the adjacency graph, and
$T_0\in\Fc^-$ and $T_q\in\Fc^+$. 
Thus, the transition
$T_0\to\cdots\to T_q$ shows that $\Fc$ is connected.
Hence, $\rho(\Fc)$ is a strongly connected family of triangular faces. 

It remains to see that $\rho(\Fc)$ touches every facet of $P_G^{\pm}$. 
By the facet-subgraph description of symmetric edge polytopes \cite[Theorem~3(2)]{CDK}, 
the directed edges corresponding to the vertices of any facet form a connected spanning subgraph of $G$. 
In particular, every facet contains some vertex $\rho(e)$ with $e$ incident with $u$. 
Every such directed edge belongs to some member of $\Fc^-\cup\Fc^+$.
Hence, $\rho(\Fc)$ touches every facet. 
By Theorem~\ref{thm:mcmullen}, $P_G^{\pm}$ is Minkowski indecomposable. 

\noindent
{\bf (If)}: The decomposability of the three exceptional families is given in Section~\ref{sec:exceptional}. 
\end{proof}

\end{document}